\documentclass{amsart}
\usepackage[thm,stdsize]{botex}

\newcommand{\cvec}{\ensuremath{\mathbf{1}}}

\newcommand{\cost}{\ensuremath{\mathrm{cost}}}

\newcommand{\opt}{\ensuremath{\mathrm{opt}}}
\newcommand{\sol}{\cost({\vec x})}
\newcommand{\energy}{\ensuremath{\mathcal{E}}}
\newcommand{\kld}{\ensuremath{\mathrm{KL}}}

\newcommand{\diag}[1]{\mathbf{Diag}(#1)}

\allowdisplaybreaks

\begin{document}
\title{On the Convergence Time of a Natural Dynamics\\ for Linear Programming}
\author{Vincenzo Bonifaci$^*$}%
\thanks{$^*$ Vincenzo Bonifaci, Institute for the Analysis of Systems and Informatics, National Research Council of Italy (IASI-CNR), Rome, Italy} 
%
%
\begin{abstract}
We consider a system of nonlinear ordinary differential equations for the solution of linear programming (LP) problems that was first proposed in the mathematical biology literature as a model for the foraging behavior of acellular slime mold \emph{Physarum polycephalum}, and more recently considered as a method to solve LPs. We study the convergence time of the continuous Physarum dynamics in the context of the linear programming problem, and derive a new time bound to approximate optimality that depends on the relative entropy between projected versions of the optimal point and of the initial point. The bound scales logarithmically with the LP cost coefficients and linearly with the inverse of the relative accuracy, establishing the efficiency of the dynamics for arbitrary LP instances with positive costs. 
\end{abstract}
\maketitle

\tableofcontents
\vfill
\newpage

\section{Introduction}
\label{sec:intro}
The theoretical analysis of natural systems has historically been the domain of mathematical biology, dynamical systems theory, and physics, but certain natural processes are capable of exhibiting remarkable information processing abilities which are often best understood from an optimization perspective. Indeed, the application of a ``computational lens'' to such processes has been advocated in different disciplines, and 
efforts are underway to identify, classify and analyze these so-called \emph{natural algorithms} \cite{Navlakha:2011,Chazelle:2012}. 

One such example can be found in the slime mold \emph{Physarum polycephalum}. \emph{P.~polycephalum} is an acellular, amoeboid slime mold in the \emph{Mycetozoa} group. In controlled experiments, the slime mold's capabilities have been leveraged to determine the shortest path between two locations in a network \cite{Nakagaki:2000,Tero:2006} and, more generally, to adaptively form efficient transport networks \cite{Tero:2010}. 
In fact, a dynamical model proposed by the mathematical biologists to describe the time evolution of \emph{P.~polycephalum}'s network physiology \cite{Tero:2007} has been rigorously proved to be algorithmically efficient for problems such as the single-source single-sink shortest path problem \cite{Bonifaci:2012,Becchetti:2013} and the minimum-cost transshipment problem \cite{Ito:2011,Straszak:2016:soda}. 

More recently, a variant of what we will call for short the \emph{Physarum dynamics} has been proposed for solving linear programming (LP) problems \cite{Johannson:2012}. 
Such dynamics is a direct mathematical extension of the one that has been studied for the shortest path and transshipment problems. 
It was shown that, under very mild assumptions on the linear program, the dynamics converges to an optimal LP solution \cite{Straszak:2016:itcs}. However, the bound for the time of convergence of a discretization of the dynamics to an approximate solution has only been proved to be polynomial when the LP cost coefficients are polynomially bounded and the constraint matrix has bounded maximum subdeterminant. 

The contribution of this paper is to study the convergence time of the continuous Physarum dynamics in the context of the linear programming problem, and to derive a new time bound to approximate optimality that does not depend on the maximum subdeterminant of the constraint matrix, and depends only logarithmically on the LP costs, establishing efficiency for any LP instance with positive costs. The proof is based on convex duality and on a dimensionless potential function that involves the relative entropy between the optimal and the current LP solution. 
We leave the study of a discretized version of the dynamics for future work, but it is natural to conjecture that some appropriate discretization should behave similarly to the continuous time dynamics. 
Indeed, several convex optimization methods can be interpreted as the discretization of an ordinary differential equation system, the solutions of which are guaranteed to converge to the set of minimizers; a well-known example is the interior point method \cite{Karmarkar:1990,Bayer:1989}. 

There is another compelling reason to study the convergence properties of the Physarum dynamics. It has been showed that, at least when started from a feasible point, this dynamics can be interpreted as a \emph{natural gradient descent} algorithm in a space endowed with a non-Euclidean metric obtained from an entropy-like function \cite{Straszak:2016:itcs}. This is also the case for certain incarnations of well-known meta-algorithms, such as Mirror Descent \cite{Nemirovski:1983,Beck:2003}, which are at the basis of very effective approximation algorithms for machine learning and convex optimization problems \cite{Arora:2012}. 
We build on the intuition from Straszak and Vishnoi \cite{Straszak:2016:itcs} by showing that, when the feasible LP region is the unit simplex, but independently of whether the initial point is feasible or not, the Physarum dynamics is identical to the continuous Mirror Descent dynamics of Nemirovski and Yudin \cite{Nemirovski:1983} in the metric generated by the negative entropy function. Thus, when the feasible LP region is the simplex, the dynamics can be interpreted as a Mirror Descent method in \emph{information geometry} \cite{Amari:2016}. However, a similar connection may not hold for a more general feasible region. A full characterization of the meta-algorithm behind the dynamics remains open; we believe it deserves to be investigated, due to its potential to suggest a novel iterative approach to linear optimization problems. 



\subsection{Linear programming and the Physarum dynamics}
\label{sec:dynamics-def}
Let $N$ and $E$ be two finite index sets. 
Given a real matrix $\vec A \in \Real^{N \times E}$, a positive vector $\vec c \in \Real^E_{>0}$, and a vector $\vec b \in \Real^N$, we consider the linear programming problem
\begin{align}
\label{eq:LP}
\min\ & \cost(\vec x) \\
\notag \text{s.t. } & \vec A \vec x = \vec b \\
\notag & \vec x \ge \vec 0, \vec x \in \Real^E 
\end{align}
where $\cost(\vec x) \defas \vec c\tp \vec x$. 
We assume that $\vec A$ has full rank and that a nonzero optimal solution to \eqref{eq:LP} exists; uniqueness is not required. We denote by $\vec x^*$ an arbitrary optimal solution to $\eqref{eq:LP}$, and denote by $\opt \defas \cost(\vec x^*)$ its value.  

We describe the \emph{(directed) Physarum dynamics} \cite{Tero:2007,Ito:2011,Bonifaci:2012,Bonifaci:2013,Becchetti:2013,Straszak:2016:soda,Straszak:2016:itcs} that solves \eqref{eq:LP}. 
Let $\vec x \in \Real^E_{> 0}$ be a positive vector, and let $\vec C$ be the diagonal matrix with entries $x_j/c_j$, for $j \in E$. Let $\vec L \defas \vec A \vec C \vec A\tp$; the matrix $\vec L$ is nonsingular and positive definite. Let $\vec p \in \Real^N$ be the unique solution to $\vec L \vec p = \vec b$, and let $\vec q \defas \vec C \vec A\tp \vec p$.  
The \emph{Physarum dynamics} for the linear program \eqref{eq:LP} is 
\begin{equation}
\label{eq:adap-cont}
\dot{x}_j(t) = q_j(t) - x_j(t)
\qquad \text{ for all } j \in E
\end{equation}
over the domain $\Omega \defas \Real^E_{>0}$, where we used the notation $\dot{x}_j(t) \defas (d/dt) x_j(t)$. 
In vector notation, and omitting the implicit dependency on time, the Physarum dynamics can be written as
\begin{equation}
\label{eq:ode-cont}
\dot{\vec x} = \vec C \vec A\tp \vec L^{-1} \vec b - \vec x. 
\end{equation}
The dynamical system has an initial condition of the form $\vec x(0)=\vec s$ for some $\vec s \in \Real^E_{>0}$. 
Existence of a solution $\vec x(t)$ to \eqref{eq:ode-cont} for $t \in [0,\infty)$ has been proved by Straszak and Vishnoi \cite[Theorem 1.1]{Straszak:2016:itcs}. The system \eqref{eq:ode-cont} is well-defined irrespective of whether the starting vector $\vec x(0)$ satisfies $\vec A \vec x(0) = \vec b$ or not; the case where it does is referred to as the \emph{feasible start} case. In the special case where $\vec A$ is derived from the signed incidence matrix of a graph, problem \eqref{eq:LP} is a \emph{minimum-cost transshipment problem} and several of the quantities defined above have an intuitive interpretation; we refer to Section \ref{sec:network} for details. 

\subsection{Our contribution}
\label{sec:our}
From previous results, it is known that the solution to \eqref{eq:ode-cont} exists, and that it converges to a feasible and optimal solution of the linear program \eqref{eq:LP}. The known bound on the convergence time, however, depends on the largest absolute value of a subdeterminant of the constraint matrix $\vec A$. 
Our main contribution is to show that, in the case of feasible start, this dependence is unnecessary, and that one can obtain a bound that only depends logarithmically on the ratio between the starting cost and the optimal cost, and on the relative entropy of the optimal solution with respect to the starting solution. More precisely, we prove the following theorem. 

\begin{theorem}
\label{thm:main-result}
For a feasible initial condition $\vec s \in \Real^E_{>0}$, consider the solution $\vec x:[0,\infty) \to \Omega$ to the Physarum dynamics \eqref{eq:ode-cont} with $\vec x(0)=\vec s$. Then $\vec x(t)$ is a feasible solution to \eqref{eq:LP} for any $t\ge 0$, and for any $\eps>0$, it holds that $\cost(\vec x(t)) \le (1+\eps) \opt$ for all 
$$ t \ge \frac{6}{\eps} \left( \ln \frac{\cost(\vec x(0))}{ \opt} + \kld(\vec \xi^*, \vec \xi(0)) \right), $$
where $\kld(\cdot,\cdot)$ denotes the relative entropy (Kullback-Leibler divergence) between distributions, and $\xi_j(0) \defas c_j x_j(0) / \cost(\vec x(0))$, $\xi_j^* \defas c_j x_j^* / \opt$ for $j \in E$. 
In particular, $\cost(\vec x(t)) \le (1+\eps) \opt$ for all 
$$ t \ge \frac{6}{\eps} \left( 2 \ln \frac{\cost(\vec x(0))}{ \opt} + \ln \mu \right), $$
where $\mu \defas \max_{j \in E} x_j^* / x_j(0)$.  
\end{theorem}

We remark that our result applies to the continuous formulation of the dynamics, and not necessarily to its Euler discretization that has been considered, together with the continuous one, in previous papers. While we conjecture that \emph{some} discretization may be similarly efficient as the bound in Theorem \ref{thm:main-result} suggests, it may also be the case that a simple Euler discretization is insufficient to obtain such a result and that a more accurate discretization technique, such as a Runge-Kutta method, would help in this sense. 

The appearance of the relative entropy term in our potential function is not an accident: by building on the geometric interpretation given by Straszak and Vishnoi \cite{Straszak:2016:itcs}, we show that when the feasible LP region is the unit simplex, independently of whether the dynamics is initialized with a feasible point or not, its trajectories coincide with those of the continuous Mirror Descent method of Nemirovski and Yudin \cite{Nemirovski:1983} in a metric with geometry dictated by the negative entropy function -- also known as the information geometry metric \cite{Amari:2016}. 

\subsection{Related work}
\label{sec:related}
An undirected variant of the Physarum dynamics has been first proposed in the mathematical biology literature by Tero, Kobayashi and Nakagaki \cite{Tero:2007} as a model for the foraging physiology of the true slime mold \emph{Physarum polycephalum}, an acellular organism that has been proved capable of solving shortest path problems effectively in laboratory experiments \cite{Nakagaki:2000}. The convergence to optimality of the continuous dynamics for the shortest path problem and for its close generalization -- the minimum-cost transshipment problem -- has been studied analytically by Bonifaci, Mehlhorn and Varma \cite{Bonifaci:2012} and by Ito et al.~\cite{Ito:2011}. An analysis of the convergence time of the Euler discretization of the dynamics was carried out by Becchetti et al.~\cite{Becchetti:2013} for the shortest path problem, and by Straszak and Vishnoi \cite{Straszak:2016:soda} for the minimum-cost transshipment problem. In summary, these works proved that the Physarum dynamics yields a polynomial-time approximation scheme to the shortest path problem and to the transshipment problem, assuming that the costs associated to the edges of the network are polynomially bounded. Observe that, in the statement of Theorem \ref{thm:main-result}, the costs are confined within logarithms, and thus a discrete version of the dynamics that achieved a similar convergence time as in Theorem \ref{thm:main-result} would \emph{not} require the costs to be polynomially bounded to be efficient.  

The generalization of the Physarum dynamics to linear programming problems that we consider here has been first suggested by Johannson and Zou \cite{Johannson:2012}. Most relevant to the current paper is the work of Straszak and Vishnoi \cite{Straszak:2016:itcs}, who initiated the rigorous study of the Physarum dynamics for LP problems of the form \eqref{eq:LP}. Straszak and Vishnoi proved that a solution to the dynamics exists over the entire time horizon $[0,\infty)$, and that a bound on the convergence time of the continuous dynamics can be expressed in terms of the parameter $\mathcal{D}$, the \emph{largest absolute value of a subdeterminant} of the constraint matrix $\vec A$, as summarized by their theorem that we quote here for comparison. 

\begin{theorem}\cite[Theorem 6.3]{Straszak:2016:itcs}
\label{thm:sv}
Suppose that $\vec x: [0,\infty) \to \Omega$ is any solution to the Physarum dynamics. Then, for some $R, \nu > 0$ depending only on $\vec A$, $\vec b$, $\vec c$, $\vec x(0)$, we have
$$ \abs{\cost(\vec x(t)) - \opt} \le R \cdot \exp(-\nu t), $$
where one can take $\nu = \mathcal{D}^{-3}$ and $R = \exp(8 \mathcal{D}^2 \cdot \norm[1]{\vec c} \cdot \norm[1]{\vec b}) \cdot (\card{E} + M_x)^2$. Here, 
$$\mathcal{D} \defas \max \{ \abs{\det(\vec A')} \,:\, \vec A' \text{ a square submatrix of } \vec A \}, $$
and 
$$ M_x \defas \max\left( \max_{j\in E} x_j(0), \max_{j\in E} x_j^{-1}(0) \right). $$
\end{theorem}
Compared to this result of Straszak and Vishnoi, our contribution is to derive a bound that avoids the dependency on $\mathcal{D}$, thus showing that the dynamics are efficient  --to the extent made precise in the statement of Theorem \ref{thm:main-result}-- for all linear programs of the form \eqref{eq:LP}, not just for those with special constraint matrices. 
Note that, in general, the bounds of Theorem \ref{thm:main-result} and \ref{thm:sv} are incomparable: for a fixed relative error $\eps$, the time for convergence guaranteed by Theorem \ref{thm:main-result} scales polynomially in the input encoding length, while Theorem \ref{thm:sv} only yields an exponential dependence; on the other hand, for a fixed input, Theorem \ref{thm:sv} achieves a polynomial dependence on $\log(1/\eps)$ (by taking $t >\!\!> \mathcal{D}^3$), while this is $O(1/\eps)$ in Theorem \ref{thm:main-result}. It is known that a simultaneous polynomial dependence on $\log(1/\eps)$ and on the input length cannot be achieved \cite[Appendix B]{Straszak:2016:soda}. 

As mentioned in the introduction, several convex optimization methods can be interpreted as discretizations of ordinary differential equation systems: for example, the Interior Point method \cite{Bayer:1989,Karmarkar:1990} and the Mirror Descent method \cite[Chapter 3]{Nemirovski:1983}. 
Straszak and Vishnoi \cite{Straszak:2016:itcs} proved that the Physarum dynamics with feasible start can be interpreted as natural gradient descent in an appropriate information metric. 
Amari \cite{Amari:2016} gives an overview of natural gradient methods in the context of information geometry; see also Raskutti and Mukherjee \cite{Raskutti:2015}.

\subsection{Organization of the paper}
The remainder of the paper is organized as follows. In Section \ref{sec:basic} we prove some basic facts about the Physarum dynamics, including an alternative characterization of the vector $\vec q \in \Real^E$ defined in Section \ref{sec:dynamics-def}. In Section \ref{sec:feasibility} we discuss the time of convergence to the feasible region of the LP and prove that the set of feasible LP solutions is an invariant set for the dynamics. In Section \ref{sec:cost} we consider the time evolution of the cost of a feasible solution and prove our main result, Theorem \ref{thm:main-result}. Finally, in Section \ref{sec:MD} we show that when the feasible LP region is the unit simplex, the Physarum dynamics can be interpreted as the continuous Mirror Descent method in a metric derived from a negative entropy function. Due to space constraints, many proofs are deferred to the Appendix. 

\section{Basic properties of the dynamics}
\label{sec:basic}
\subsection{Notation} In the paper we reserve boldface symbols for vectors or matrices and non-boldface symbols for scalars or sets. 
We use the standard norms: for example, for $\vec v \in \Real^n$: $\norm[1]{\vec v} \defas \sum_{i=1}^n \abs{v_i}$, $\norm[2]{\vec v} \defas (\sum_{i=1}^n v_i^2)^{1/2}$. With the notation $\diag{(d_i)_{i=1}^n}$ we mean the $n \times n$ diagonal matrix with $d_i$ as the $i$th term on the main diagonal. 

For the whole paper, the linear program \eqref{eq:LP} is fixed, in other words the triple $(\vec A, \vec b, \vec c)$ is fixed. Whenever the matrices or vectors $\vec C = \diag{(x_j/c_j)_{j\in E}}$, $\vec L=\vec A \vec C \vec A\tp$, $\vec p = \vec L^{-1} \vec b$, or $\vec q = \vec C \vec A\tp \vec p$ appear, they should be understood as computed with respect to a point $\vec x \in \Real^E_{>0}$. As $\vec x = \vec x(t)$ evolves in time with the dynamics \eqref{eq:ode-cont}, the former quantities are time-varying as well. The quantity $\energy \defas \vec q\tp \vec R \vec q$ is called the \emph{energy} of the vector $\vec q$. 

For a strictly convex and differentiable function $\psi : \Omega \to \Real$, the \emph{Bregman divergence} under $\psi$ is the function 
$$  
D_{\psi}(\vec x', \vec x) \defas \psi(\vec x') - \psi(\vec x) - \nabla \psi(\vec x)\tp \cdot (\vec x' - \vec x), 
$$
where $\vec x'$, $\vec x \in \Omega$. 
The Bregman divergence is in general not symmetric, but it is nonnegative and satisfies $D_\psi(\vec x', \vec x) = 0$ if $\vec x'=\vec x$. 
The \emph{Legendre dual} of $\psi$ is the function $\psi^\star: \Real^E \to \Real$ defined by
$$ \psi^\star(\vec y) \defas \sup_{\vec x \in \Real^E} (\vec x \tp \vec y - \psi(\vec x)), $$
Note that a vector $\vec x \in \Omega$ is a maximizer of $\vec x\tp \vec y - \psi(\vec x)$ iff $\vec y = \nabla \psi(\vec x)$. 

\subsection{Intuition: The network case}
\label{sec:network}
The interpretation of the dynamics defined in Section \ref{sec:dynamics-def} in the case where the constraint matrix is a network matrix is particularly appealing, as most of the statements below have physical interpretations in that case. Indeed, when $\vec A$ is derived from the (signed) node-edge incidence matrix of a graph with node set $N$ and edge set $E$ \footnote{More precisely, since we stipulated that $\vec A$ should be full rank, we omit from $N$ one of the nodes and omit the corresponding row from $\vec A$; this corresponds to ``grounding'' the potential value of this node to zero.}, the dynamics \eqref{eq:ode-cont} have a natural interpretation in terms of electrical networks: the vector $\vec b$ prescribes the external in-flow of current at each node, the matrix $\vec L$ is the (reduced) graph Laplacian, the vector $\vec p$ defines the Kirchhoff node potentials, the vector $\vec q$ is the electrical flow, and $\energy$ is the energy dissipation (per unit time) of the network. In this context, Lemma \ref{lem:thomson} below is nothing but the \emph{principle of least action} for electrical networks, also known as Thomson's principle, stating that the electrical flow is the feasible flow that minimizes energy dissipation \cite[Theorem IX.2]{Bollobas:1998}.
The duality relation \eqref{eq:duality} becomes Ohm's law. 
Proposition \ref{prop:pLp} is the conservation of energy principle, stating that if one replaces a network with a current source $s$ and a sink $\bar{s}$ with a single wire whose resistance is the effective resistance of the network, then the total energy in the system does not change \cite[Theorem IX.3]{Bollobas:1998}. 
Proposition \ref{lem:energy} is known as Tellegen's theorem. Of course, the difference with classical circuit theory is that the resistor values are dynamically adjusted in response to the flow: the Physarum dynamics adjusts the edges' resistances $c_j/x_j$, by updating the $x_j$ via \eqref{eq:adap-cont}. In the network case, the dynamics converges to the solution of a minimum-cost transshipment problem with cost function prescribed by $\vec c$ and node demands/supplies prescribed by $\vec b$ (see for example \cite[Theorem 1.2]{Straszak:2016:soda}). However, we remark that in the following statements we never need the fact that $\vec A$ is derived from a network matrix; our results hold more generally for any full-rank matrix. 

\subsection{Basic properties of the dynamics}
We start by giving an alternative characterization of the vector $\vec q \defas \vec C \vec A\tp \vec L^{-1} \vec b$. 
\begin{lemma}
\label{lem:thomson}
The vector $\vec q \in \Real^E$ defined in Section \ref{sec:dynamics-def} equals the unique optimal solution to the continuous quadratic optimization problem: 
\begin{align}
\label{eq:thomson}
\min\ & \vec f\tp \vec R \vec f \\
\notag \text{s.t. } & \vec A \vec f = \vec b. 
\end{align}
where $\vec R \defas \vec C^{-1} \in \Real^{E \times E}$ is the diagonal matrix with value $r_j \defas c_j/x_j$ for the $j$-th element of the main diagonal. Moreover, 
\begin{equation}
\label{eq:duality}
\vec R \vec q = \vec A\tp \vec p.
\end{equation}
\end{lemma}
\begin{proof}
To establish \eqref{eq:duality}, simply left-multiply with $\vec R$ the identity $\vec q = \vec C \vec A\tp \vec p$. 
It remains to establish the first part of the claim. 
Since the objective function in \eqref{eq:thomson} is strictly convex, the problem has a unique optimal point. 
Consider any feasible point $\vec f$, and define $\vec g= \vec f - \vec q$. Then $\vec A \vec g = \vec b - \vec b = \vec 0$ and hence
$$
\vec f\tp \vec R \vec f = (\vec q + \vec g)\tp \vec R (\vec q + \vec g) =
\vec q\tp \vec R \vec q + 2 \vec g\tp \vec R \vec q + \vec g\tp \vec R \vec g 
\ge 
\vec q\tp \vec R \vec q,  
$$
since $\vec g\tp \vec R \vec g \ge 0$ and $\vec g\tp \vec R \vec q = \vec g\tp \vec A\tp \vec p = (\vec A \vec g)\tp \vec p = \vec 0\tp \vec p = 0$. Therefore, the objective function value of any feasible point $\vec f$ is at least as large as the objective function value of $\vec q$. 
\end{proof}


\begin{proposition}
\label{prop:pLp}
$\energy = \vec b\tp \vec L^{-1} \vec b = \vec p\tp \vec L \vec p$. 
\end{proposition}
\begin{proof}
$\vec q\tp \vec R \vec q = (\vec b\tp \vec L^{-1} \vec A \vec C) \vec R (\vec C \vec A\tp \vec L^{-1} \vec b) = (\vec b\tp \vec L^{-1}) (\vec A \vec C \vec A\tp) (\vec L^{-1} \vec b) = $ \\ $ = \vec p\tp \vec L \vec p$. 
\end{proof}

\begin{proposition}
\label{lem:energy}
Let $\vec f$ satisfy $\vec A \vec f = \vec b$. Then
\begin{equation}
\label{eq:energy}
\vec f\tp \vec A\tp \vec p(t) = \energy(t). 
\end{equation}
\end{proposition}
\begin{proof}
Since $\vec A \vec f=\vec b$, we have 
$ \vec p\tp \vec A \vec f = \vec p\tp \vec b = \vec p\tp \vec L \vec p = \energy. $
The last equality is due to Proposition \ref{prop:pLp}. 
\end{proof}

\section{Convergence to the feasible region}
\label{sec:feasibility}
In this section we discuss the time of convergence to the feasible region $\vec A \vec x = \vec b$. In particular, we aim to show that feasibility is invariant under the dynamics: a feasible starting point remains feasible at all times. It turns out that a stronger property holds: the Euclidean norm of the ``infeasibility error'' $\vec e \defas \vec A \vec x - \vec b$ approaches zero exponentially fast (Lemma \ref{lem:convergence-to-feasibility}). 

\begin{proposition}
\label{prop:feasibility}
$\vec A \dot{\vec x} = \vec b - \vec A \vec x. $
\end{proposition}
\begin{proof}
Using the definition of the dynamics \eqref{eq:ode-cont}, 
$\vec A \dot{\vec x} = \vec A \vec C \vec A\tp \vec L^{-1} \vec b - \vec A \vec x = \vec L \vec L^{-1} \vec b - \vec A \vec x = \vec b - \vec A \vec x. $
\end{proof}

\begin{lemma}
\label{lem:convergence-to-feasibility}
Let $\vec e(t) \defas \vec A \vec x(t) - \vec b$. Then $\norm[2]{\vec e(t)} = \norm[2]{\vec e(0)} \exp{(-t)}$ for any $t>0$. 
In particular, if $\vec A \vec x(0)=\vec b$ then $\vec A \vec x(t)=\vec b$ for all $t>0$. 
\end{lemma}
\begin{proof}
We have
\begin{align*}
\frac{d}{dt} \norm[2]{\vec e}^2 &= \frac{d}{dt} (\vec A \vec x- \vec b)\tp (\vec A \vec x- \vec b) \\
&= \frac{d}{dt} \left( \vec x\tp \vec A\tp \vec A \vec x - 2 \vec b\tp \vec A \vec x + \vec b\tp \vec b \right) = \\
&= 2 \vec x\tp \vec A\tp \vec A \dot{\vec x} - 2 \vec b\tp \vec A \dot{\vec x} \\
&=2 (\vec x\tp \vec A\tp - \vec b\tp) \vec A \dot{\vec x} \\
&=2 (\vec A \vec x - \vec b)\tp (\vec b - \vec A \vec x) \\
&= -2 \norm[2]{\vec e}^2,   
\end{align*}
where we used Proposition \ref{prop:feasibility}. 
Solution of the differential equation above yields $\norm[2]{\vec e(t)}^2 = \norm[2]{\vec e(0)}^2 \exp{(-2t)}.$ Taking square roots yields the claim. 
\end{proof}


\section{Convergence in cost value}
\label{sec:cost}
To analyze the convergence in cost values, and eventually prove Theorem \ref{thm:main-result}, it will be useful to consider normalized versions of the candidate solution $\vec x(t)$ and of the optimal vector $\vec x^*$. 
For any $j \in E$, let $\xi_j(t) \defas c_j x_j(t) / \cost(\vec x(t))$, $\xi_j^* \defas c_j x_j^* / \opt$. 
%
Then, by construction, $\cvec\tp \vec \xi^* = \cvec\tp \vec \xi(t) = 1$, $\vec \xi(t) > \vec 0$ and $\vec \xi^* \ge \vec 0$, so $\vec \xi(t)$ and $\vec \xi^*$ can be interpreted as probability distributions over $E$. The \emph{relative entropy} of $\vec \xi^*$ with respect to $\vec \xi$, or \emph{Kullback-Leibler divergence} $\kld(\vec \xi^*, \vec \xi(t))$, is defined as: 
$$ \kld(\vec \xi^*, \vec \xi(t)) \defas \sum_{j \in E} \xi_j^* \ln \frac{\xi_j^*}{\xi_j(t)}. $$
The KL divergence is the Bregman divergence of the negative entropy function $\vec x \mapsto \sum_j x_j \ln x_j$; it is always nonnegative, and it is zero iff $\vec \xi^* = \vec \xi(t)$ (see for example \cite[Chapter 1]{Amari:2016}). 

We can now define the potential function that is central to our analysis. Let
\begin{equation}
\label{eq:phi}
 \Phi(t) \defas \ln \frac{\cost(\vec x(t))}{\opt} + \kld(\vec \xi^*, \vec \xi(t)). 
\end{equation}
Note that the first term is nonnegative whenever $\vec x(t)$ is feasible for the LP, and the second term is always nonnegative. Similar to previous analysis of Physarum dynamics based on potential functions \cite{Becchetti:2013,Straszak:2016:soda,Straszak:2016:itcs}, the potential function $\Phi$ contains a term that depends on the cost of the candidate solution $\vec x$, and an ``entropic barrier'' term that captures the geometry of the feasible region: in particular, the second term penalizes distributions that get too close to the boundary of the positive orthant whenever the corresponding coordinate of the optimal solution is not on the boundary (that is, $\xi_j(t) \approx 0$ but $\xi_j^*>0$). A difference with respect to previous papers is that the potential function \eqref{eq:phi} is dimensionless, which is natural since our aim is to bound the relative, rather than  absolute, approximation error. 

To proceed further, we study the evolution of the potential function over time. We start by bounding the derivative of various terms that compose it. 
\begin{lemma}
\label{lem:cost-vs-opt}
For any $\vec x(t) \in \Real^E_{>0}$, 
\begin{equation}
\label{eq:cost-vs-opt}
 \frac{d}{dt} \ln \frac{\cost(\vec x(t))}{\opt} \le {\left(\frac{\energy(t)}{\cost(\vec x(t))}\right)}^{1/2} - 1. 
\end{equation}
\end{lemma}
\begin{proof}
\begin{align*}
\frac{d}{dt} \ln \frac{\sol}{\opt} &= \frac{\frac{d}{dt} \sol}{\sol} \\
&= \frac{\vec c\tp \dot{\vec x}}{\sol} \\
&\stackrel{\eqref{eq:adap-cont}}{=} \frac{\vec c\tp (\vec q - \vec x)}{\sol} \\
&= \frac{\vec c\tp \vec q}{\sol} - 1 \\
&\stackrel{(*)}{=} \frac{\sum_{j \in E} r_j q_j x_j}{\sol} - 1 \\
&\stackrel{(**)}{\le} \frac{ \left( \sum_{j \in E} r_j q_j^2 \right)^{1/2} \left( \sum_{j \in E} r_j x_j^2 \right)^{1/2}}{\sol} - 1 \\
&= \frac{{(\energy \sol)}^{1/2}}{\sol} - 1, 
\end{align*}
where in the third equality we used the definition of the dynamics, in (*) we used $r_j = c_j/x_j$, and in (**) we used the Cauchy-Schwarz inequality. For the last equality, we used the definition of the energy $\energy=\vec q\tp \vec R \vec q$ and (once more) the fact $r_j = c_j/x_j$. 
\end{proof}

The following lemma is instrumental in bounding the time derivative of the KL divergence term in \eqref{eq:phi}. 
\begin{lemma}
\label{lem:x-log-x}
For any $\vec x(t) \in \Real^E_{>0}$, 
\begin{equation}
\label{eq:x-log-x}
\frac{d}{dt} \sum_{j \in E} \frac{c_j x_j^*}{\opt} \ln \frac{x_j^*}{x_j(t)} = 1 - \frac{\energy(t)}{\opt}. 
\end{equation}
\end{lemma}
\begin{proof}
We start by computing
$$
\sum_j \frac{c_j}{\opt} x_j^* \ln \frac{x_j^*}{x_j} = - \frac{1}{\opt} \sum_j c_j x_j^* \ln x_j + \frac{1}{\opt} \sum_j c_j x_j^* \ln x_j^*. 
$$
The second term above is constant, so
\begin{align*}
\frac{d}{dt} \sum_j \frac{c_j}{\opt} x_j^* \ln \frac{x_j^*}{x_j} &= - \frac{1}{\opt} \sum_j c_j x_j^* \frac{\dot{x}_j}{x_j} = \\
&= -\frac{1}{\opt} \sum_j c_j x_j^* \frac{q_j-x_j}{x_j} = \\
&= 1- \frac{1}{\opt} \sum_j r_j q_j x_j^* = \\
&= 1- \frac{1}{\opt} \vec x^{*\top} \vec R \vec q = \\
&\using{eq:duality}{=} 1- \frac{1}{\opt} \vec x^{*\top} \vec A\tp \vec p = \\
&\using{eq:energy}{=} 1- \frac{\energy}{\opt}. 
\end{align*}
We used \eqref{eq:duality} (Lemma \ref{lem:thomson}) and the alternative characterization of the energy \eqref{eq:energy} given by Proposition \ref{lem:energy}. 
\end{proof}



\begin{lemma}
\label{lem:kl}
For any $\vec x(t) \in \Real^E_{>0}$, 
$$ 
\frac{d}{dt} \kld(\vec \xi^*, \vec \xi(t)) \le 
\left({\frac{\energy(t)}{\cost(\vec x(t))}}\right)^{1/2} - \frac{\energy(t)}{\opt} 
$$
\end{lemma}
\begin{proof}
\begin{align*}
\frac{d}{dt} \sum_j \frac{c_j x_j^*}{\opt} \ln \frac{c_j x_j^* / \opt}{c_j x_j / \sol} &= 
\frac{d}{dt} \sum_j \frac{c_j x_j^*}{\opt} \ln \frac{x_j^*}{x_j} + \frac{d}{dt} \sum_j \frac{c_j x_j^*}{\opt} \ln \frac{\sol}{\opt} \\
&\using{eq:x-log-x}{=} 1 - \frac{\energy}{\opt}  + \frac{d}{dt} \ln \frac{\sol}{\opt} \cdot 1 \\
&\using{eq:cost-vs-opt}{\le} 1 - \frac{\energy}{\opt} + \left(\frac{\energy}{\sol}\right)^{1/2} - 1 \\
&= \left({\frac{\energy}{\sol}}\right)^{1/2} - \frac{\energy}{\opt}. 
\end{align*}
In the second equality we used Lemma \ref{lem:x-log-x}; for the inequality we applied Lemma \ref{lem:cost-vs-opt}. 
\end{proof}

We are ready to prove that the more expensive a solution is, the more the potential function has to decrease. 
\begin{lemma}
\label{lem:improvement}
If $\cost(\vec x(t)) \ge (1+\eps)^2 \opt$ for some $\eps \in (0,1/2)$, then 
$(d/dt) \Phi(t) \le -\eps/2. $ 
\end{lemma}
\begin{proof}
Let $\gamma \defas \energy/\opt$ and $\delta\defas 1/(1+\eps)$. 
Combining Lemma \ref{lem:cost-vs-opt} and Lemma \ref{lem:kl} yields
\begin{align*}
\frac{d}{dt} \Phi(t) &= 2 \left({\frac{\energy}{\sol}}\right)^{1/2} - 1  - \frac{\energy}{\opt} \\
&\stackrel{(*)}{\le}  2\delta \gamma^{1/2} -1 - \gamma =\\
&= -2(1-\delta) \gamma^{1/2} - (1-\gamma^{1/2})^2, 
\end{align*}
where in $(*)$ we used the assumption $\cost(\vec x) \ge (1+\eps)^2 \opt$. 
Note that both summands in the last expression are negative. 
We distinguish two cases. If $(1-\gamma^{1/2})^2 \ge \eps/2$, then by ignoring the first summand above we obtain
$$ \frac{d}{dt}{\Phi} \le -(1-\gamma^{1/2})^2 \le -\eps/2, $$
which proves the claim. 
Otherwise, if $(1-\gamma^{1/2})^2 < \eps/2$, then $\gamma^{1/2} > 1-(\eps/2)^{1/2}$ and by ignoring the second summand we obtain
$$ \frac{d}{dt}{\Phi} \le -2(1-\delta)(1-(\eps/2)^{1/2}) \le -2 \cdot \frac{1-(\eps/2)^{1/2}}{1+\eps} \eps \le -2 \frac{1/2}{3/2} \eps < -\eps/2. $$
\end{proof}

The next lemma ensures that, for feasible solutions, the energy is always a valid lower bound on the cost. 
\begin{lemma}
\label{lem:energy-vs-cost}
Suppose $\vec x(t) \ge 0$, $\vec A \vec x(t) = \vec b$. Then $\energy(t) \le \cost(\vec x(t))$. 
\end{lemma}
\begin{proof}
By the assumption, $\vec x(t)$ is a feasible LP solution. 
By Lemma \ref{lem:thomson}, $\vec q(t)$ is a minimizer of the quadratic form $\vec f\tp \vec R \vec f$ among all vectors $\vec f$ satisfying $\vec A \vec f = \vec b$. One possible such vector is $\vec x$. Thus, 
\begin{equation}
\label{eq:energy-vs-cost}
\energy = \vec q\tp \vec R \vec q \le \vec x\tp \vec R \vec x = \sum_{j\in E} \frac{c_j}{x_j} x_j^2 = \sol. 
\end{equation}
\end{proof}

As a corollary, by Lemma \ref{lem:cost-vs-opt} the cost of a feasible solution does not increase over time. 
\begin{corollary}
\label{cor:monotone-cost}
Suppose $\vec x(t) \ge 0$, $\vec A \vec x(t) = \vec b$. Then $(d/dt) \cost(\vec x(t)) \le 0$. 
\end{corollary}
\begin{proof}
Combine Lemma \ref{lem:energy-vs-cost} and Lemma \ref{lem:cost-vs-opt}. 
\end{proof}

All ingredients are now into place to derive our main claim, from which Theorem \ref{thm:main-result} will directly follow. 
\begin{theorem}
\label{thm:main}
Suppose $\vec x(0) > 0$, $\vec A \vec x(0)=\vec b$. Then 
\begin{enumerate}
\item[(a)]
$\vec x(t)$ is feasible for LP \eqref{eq:LP} for any $t \ge 0$; 
\item[(b)]
$\cost(\vec x(t)) \le (1+\eps)\opt$ for all $$ t \ge \frac{\Phi(0)}{\eps/6} = \frac{6}{\eps} \bigg( \ln \frac{\cost(\vec x(0)) }{ \opt} + \kld(\vec \xi^*, \vec \xi(0))  \bigg). $$
\end{enumerate}
\end{theorem}
\begin{proof}
By assumption, we start with a feasible initial solution $\vec x(0)$, thus by Lemma \ref{lem:convergence-to-feasibility} the solution $\vec x(t)$ stays feasible for all $t \ge 0$; this proves point (a). By Corollary \ref{cor:monotone-cost}, the cost of $\vec x(t)$ can only decrease as $t$ increases. To prove point (b), assume, by contradiction, that $\cost(\vec x(t_0))$ is larger than $(1+\eps)\opt$ for some $t_0$ that is larger than $\Phi(0)/(\eps/6)$. 
By Lemma \ref{lem:improvement}, $(d/dt){\Phi}(t) \le -\eps/6$ for all $t$ such that $$\cost(\vec x(t)) \ge (1+\eps) \opt = (1+3\eps') \opt \ge (1+\eps')^2 \opt,$$ where $\eps' \defas\eps/3$. In particular, $(d/dt) \Phi(t) \le -\eps/6$ would hold for all $t \in [0,t_0]$. 
This implies the desired contradiction, since $\Phi(t_0) = \Phi(0) + \int_0^{t_0} \frac{d}{dt} \Phi(t) \le \Phi(0) - (\eps/6) t_0$ would have to be negative. This is impossible since $\vec x(t)$ is feasible at all times and thus $\Phi(t)$ is nonnegative for all $t$.  
\end{proof}

\begin{proof}[Proof of Theorem \ref{thm:main-result}]
Theorem \ref{thm:main} already proves the first part of Theorem \ref{thm:main-result}. For the second part, observe that if $$\mu \defas \max_{j \in E} \frac{x_j^*}{x_j(0)}, $$ then 
\begin{align*}
\kld(\vec \xi^*, \vec \xi(t)) &= 
\sum_{j \in E} \frac{c_j x_j^*}{\opt} \ln \frac{c_j x_j^* / \opt}{c_j x_j / \sol} \\
&= \sum_j \frac{c_j x_j^*}{\opt} \ln \frac{x_j^*}{x_j} + \sum_j \frac{c_j x_j^*}{\opt} \ln \frac{\sol}{\opt} \\
& \le (\ln \mu) \cdot \sum_j \frac{c_j x_j^*}{\opt} + \left( \ln \frac{\sol}{\opt} \right) \cdot \sum_j \frac{c_j x_j^*}{\opt} \\
& = \ln  \mu + \ln \frac{\sol}{\opt}.  
\end{align*}
Substitution in the bound of Theorem \ref{thm:main} yields the claim. 
\end{proof}

\section{Physarum dynamics for the unit simplex}
\label{sec:MD}
In this final section we connect the Physarum dynamics to a classical convex optimization method and provide an intuitive justification for the relative entropy term in our potential function \eqref{eq:phi}. We prove that when the feasible LP region is the unit simplex, the trajectories of \eqref{eq:ode-cont} coincide (even \emph{outside} the simplex) with those of the continuous Mirror Descent method in an information metric. 

Nemirovksi and Yudin \cite{Nemirovski:1983} defined Mirror Descent as a system of ordinary differential equations aiming to minimize a convex function $F: \Omega \to \Real$ (for some $\Omega \subseteq \Real^n$). The Mirror Descent system has the form
\begin{align}
\label{eq:md}
\vec x &= \nabla \psi^\star (\vec y) \\
\notag \dot{\vec y} &= -\nabla F(\vec x), 
\end{align}
where $\psi$ is a strictly convex function on $\Omega$ and $\psi^\star$ is its Legendre dual. The vector $\vec y \in \Real^n$ is connected to the primal vector $\vec x$ by the invertible relation $\vec x = \nabla \psi^\star (\vec y)$. The dynamics is thus defined on the dual variables $\vec y$ and ``reflected'' in the primal variables $\vec x$. 

In the Physarum setting, $\Omega = \Real^E_{>0}$, and we take $\psi$ to be the negative entropy function $\psi(\vec x) \defas \sum_{j\in E} x_j \ln x_j$. This allows us to prove the following. 

\begin{theorem}
\label{thm:md}
Let $\vec A=\vec 1\tp$, $\vec b= 1$ in \eqref{eq:LP}, and let $F(\vec x) \defas \vec 1\tp \vec x + \ln \energy(\vec x)$. Then, the dynamics \eqref{eq:adap-cont} coincide with the Mirror Descent dynamics \eqref{eq:md}. 
\end{theorem}
\begin{proof}
A proof of the convergence of \eqref{eq:md} to optimality is based on a potential function $V(\vec x(t))$ that equals the Bregman divergence, under the function $\psi$, from the optimal point $\vec x^*$ to the current primal point: 
\begin{equation}
\label{eq:bregman-lyapunov}
V(\vec x(t)) \defas D_{\psi}(\vec x^*, \vec x(t)). 
\end{equation}
%
The Bregman divergence under $\psi$ is just the (unnormalized) relative entropy between $\vec x^*$ and $\vec x(t)$:
\begin{align}
D_\psi(\vec x^*, \vec x(t)) &= \psi(\vec x^*) - \psi(\vec x(t)) -\nabla \psi(\vec x(t))\tp \cdot (\vec x^* - \vec x(t)) \\
\notag &= \sum_{j \in E} x^*_j \ln x^*_j - \sum_{j \in E} x_j(t) \ln x_j(t) + {} \\
\notag & \mathrel{\phantom{bla bla bla}} \quad{} - \sum_{j \in E} (1+\ln x_j(t)) (x^*_j - x_j(t)) \\
\notag &= \sum_j x^*_j \ln \frac{x^*_j}{x_j(t)} + \sum_j x_j(t) - \sum_j x^*_j. 
\end{align}
The Legendre dual of $\psi$ is $$\psi^\star(\vec y) = \sum_{j \in E} \exp(y_j-1).$$ 
The primal-dual connection prescribed by \eqref{eq:md} is then $y_j = 1+\ln x_j$ for each $j \in E$. When the feasible LP region is the unit simplex (that is, $\vec A=\vec 1\tp$, $\vec b=1$ in \eqref{eq:LP}), the relevant quantities for the Physarum dynamics are easily computed as: 
\begin{align*}
\vec L &= \vec 1\tp \vec C \vec 1 = \sum_{j \in E} x_j/c_j, \\
\energy &= \vec b \vec L^{-1} \vec b = \left(\sum_{j \in E} x_j/c_j \right)^{-1}, \\
q_j &= \frac{x_j}{c_j} \cdot \energy \quad \text{ for } j \in E, 
\end{align*}
so the Physarum dynamics \eqref{eq:ode-cont} is 
\begin{equation}
\dot{x}_j =  \frac{x_j}{c_j} \cdot \energy  - x_j \quad \text{ for } j \in E. 
\end{equation}
By a change of variables according to the primal-dual correspondence $y_j=1+\ln x_j$, we obtain the dual-space version of the dynamics,
\begin{equation}
\label{eq:md-gradient}
\dot{y}_j = \frac{1}{c_j} \energy(\vec x) - 1,
\end{equation}
with $\energy(\vec x) = (\sum_j x_j/c_j)^{-1} = (\sum_j \exp(y_j-1)/c_j)^{-1}$. We observe that this is just the Mirror Descent method \eqref{eq:md} with objective function $ F(\vec x) \defas \vec 1\tp \vec x + \ln \energy(\vec x)$. In fact, we compute
\begin{align}
(\nabla F(\vec x))_j &= 1 + \frac{\partial \energy(\vec x)}{\partial x_j} \cdot \energy^{-1}(\vec x) \\
\notag &= 1 - \frac{1}{c_j} \left(\sum_{k \in E} \frac{x_k}{c_k}\right)^{-2} \cdot \left(\sum_{k \in E} \frac{x_k}{c_k} \right) \\
\notag &= 1 - \frac{1}{c_j} \energy(\vec x), 
\end{align}
and then \eqref{eq:md-gradient} is nothing but $\dot{\vec y} = -\nabla F(\vec x)$. 
\end{proof}

Note that we do not require the initial point $\vec x(0)$ to be feasible. 
\subsubsection*{Acknowledgment.} 
The author would like to thank Kurt Mehlhorn for suggesting a shorter proof of Lemma \ref{lem:thomson}. 


\end{document}